\documentclass[10pt]{amsart}
\usepackage{amsmath}
\usepackage{amsfonts}

\newtheorem{theorem}{Theorem}[section] 
\newtheorem{lemma}[theorem]{Lemma}     
\newtheorem{corollary}[theorem]{Corollary}
\newtheorem{proposition}[theorem]{Proposition}

\theoremstyle{definition}
\newtheorem{definition}[theorem]{Definition}

\theoremstyle{remark}
\newtheorem{remark}[theorem]{Remark}
\numberwithin{equation}{section}

\newcommand{\C}{\mathbb{C}}
\newcommand{\R}{\mathbb{R}}
\newcommand{\N}{\mathbb{N}}
\newcommand{\de}{\partial}
\newcommand{\eps}{\varepsilon}

\title[Skew Carleson measures]
 {Skew Carleson measures in strongly pseudoconvex domains} 

\author[Marco Abate]{Marco Abate}
\address{Marco Abate\\ Dipartimento di Matematica\\ Universit\`a di Pisa\\ Largo Pontecorvo 5, 56127 Pisa\\ Italy.} \email{marco.abate@unipi.it}
\author[Jasmin Raissy]{Jasmin Raissy*}
\address{Jasmin Raissy\\ Institut de Math\'ematiques de Toulouse, UMR5219\\ Universit\'e de Toulouse, CNRS\\ UPS IMT, 118 route de Narbonne, F-31062 Toulouse Cedex 9\\ France} 
\email{jraissy@math.univ-toulouse.fr}
\thanks{2010 Mathematics Subject Classification: 32A36 (primary), 32A25, 32Q45, 32T15, 46E22, 46E15, 47B35 (secondary).}
\thanks{\textit{Keywords:} Carleson measure; Toeplitz operator; strongly pseudoconvex domain; weighted Bergman space}
\thanks{$^{*}$Partially supported by ANR
 project LAMBDA,  ANR-13-BS01-0002 and by
 the FIRB2012 grant ``Differential Geometry and Geometric Function Theory'', RBFR12W1AQ 002.}

\begin{document}

\begin{abstract}
Given a bounded strongly pseudoconvex domain $D$ in $\C^n$ with smooth boundary, we give a characterization through products of functions in weighted Bergman spaces of $(\lambda,\gamma)$-skew Carleson measures on $D$, with $\lambda>0$ and $\gamma>1-\frac{1}{n+1}$.

\end{abstract}

\maketitle

\section{Introduction}

Carleson measures are a powerful tool and an interesting object to study. They have been introduced by Carleson \cite{C} in his celebrated solution of the corona problem to study the structure of the Hardy spaces of the unit disc $\Delta\subset\C$. Let $A$ be a Banach space of holomorphic functions on a domain $D\subset\C^n$; given $p\ge 1$, a finite positive Borel measure $\mu$ on~$D$ is a \emph{Carleson measure} of~$A$ (for~$p$)
if there is a continuous inclusion $A\hookrightarrow L^p(\mu)$, that is there exists a constant $C>0$ such that
\[
\forall f\in A \qquad \int_D |f|^p\,d\mu\le C\|f\|_A^p\;.
\]

In this paper, we are interested in Carleson measures for Bergman spaces, that is spaces of $L^p$ holomorphic functions, usually denoted by $A^p$ (relationships between Carleson measures for Hardy spaces and Carleson measures for Bergman spaces can be found in~\cite{Am}). Carleson measures for Bergman spaces have been studied by several authors, including Hastings \cite{H} (see also Oleinik and Pavlov \cite{OP} and Oleinik \cite{O}) for the Bergman spaces~$A^p(\Delta)$, Cima and Wogen \cite{CW} in the case of the unit ball $B^n\subset\C^n$, Zhu \cite{Zh1} in the case of bounded symmetric domains, Cima and Mercer \cite{CM} for Bergman spaces in strongly pseudoconvex domains $A^p(D)$, and Luecking \cite{Luecking} for more general domains. 

Given $D\subset\subset\C^n$ a bounded strongly pseudoconvex domain in $\C^n$ with smooth $C^\infty$ boundary, a positive finite Borel measure $\mu$ on~$D$ and $0<p<+\infty$, we denote by $L^p(\mu)$ the set of complex-valued $\mu$-measurable functions $f\colon D\to\C$ such that
\[
\|f\|_{p,\mu}:=\left[\int_D |f(z)|^p\,d\mu(z)\right]^{1/p}<+\infty\;.
\]
If $\mu=\delta^\alpha\nu$ for some $\alpha\in\R$, where $\delta(z)=d(z,\partial D)$ is the distance from the boundary of $D$ and $\nu$ is the Lebesgue measure, the \emph{weighted Bergman space} is defined as
\[
A^p(D,\alpha)=L^p(\delta^\alpha \nu)\cap\mathcal{O}(D)\;,
\]
where $\mathcal{O}(D)$ denotes the space of holomorphic functions on~$D$, endowed with the topology of uniform convergence on compact subsets. Together with Saracco, we gave in \cite{ARS} a characterization of Carleson measures of weighted Bergman spaces in terms of the intrinsic Kobayashi geometry of the domain. 

It is a natural question to study Carleson measures for different exponents, that is the embedding of weighted Bergman spaces $A^p(D, \alpha)$ into $L^q$ spaces. Given, $0<p$, $q<+\infty$ and $\alpha>-1$, a finite positive Borel measure $\mu$ is called a \emph{$(p,q;\alpha)$-skew Carleson measure} if $A^p(D,\alpha)\hookrightarrow L^q(\mu)$ continuously, that is there exists a constant $C>0$ such that
\[
\int_D |f(z)|^q\,d\mu(z)\le C \|f\|_{p,\alpha}^q
\]
for all $f\in A^p(D,\alpha)$. 
Investigation on $(p,q;\alpha)$-skew Carleson measure has been started by Luecking in \cite{Luecking2} and recently extended by Hu, Lv and Zhu in \cite{HLZ}, where these measures are called $(p,q,\alpha)$ Bergman Carleson measures. It turns out (see \cite{HLZ} and the next section for details) that the property of being $(p,q;\alpha)$-skew Carleson depends only on the quotient $q/p$ and on $\alpha$, allowing us to define \emph{$(\lambda,\gamma)$-skew Carleson} measures for $\lambda>0$ and $\gamma>1-\frac{1}{n+1}$. Roughly speaking, a measure is $(\lambda,\gamma)$-skew Carleson if and only if it is a $(p,q;(n+1)(\gamma-1))$-skew Carleson measure for some (and hence any) $p,q$ such that $q/p =\lambda$ (see Definition \ref{def:sCdef}).

\smallskip
The main result of this paper gives a characterization of $(\lambda,\gamma)$-skew Carleson measures on bounded  strongly pseudoconvex domains through products of functions in weighted Bergman spaces.

\begin{theorem}
\label{maintheorem}
Let $D\subset\subset\C ^n$ be a bounded strongly
pseudoconvex domain, and let $\mu$ be a positive finite Borel measure on $D$. Fix an integer $k\ge 1$, and let $0<p_j$,~$q_j<\infty$ and $1-\frac{1}{n+1}<\theta_j$ be given, for $j=1,\ldots, k$. Set
\[
\lambda=\sum_{j=1}^k\frac{q_j}{p_j}\qquad\mathit{and}\qquad\gamma=\frac{1}{\lambda}\sum_{j=1}^k \theta_j\frac{q_j}{p_j}\;.
\]
Then $\mu$ is a $(\lambda,\gamma)$-skew Carleson measure if and only if there exists $C>0$ such that 
\begin{equation}
\int_D \prod_{j=1}^k |f_j(z)|^{q_j}\,d\mu(z)\le C \prod_{j=1}^k \|f_j\|_{p_j,(n+1)(\theta_j-1)}^{q_j}\label{eq:main}
\end{equation}
for any $f_j\in A^{p_j}\bigl(D,(n+1)(\theta_j-1)\bigr)$. 
\end{theorem} 

This result generalizes the analogue one obtained by Pau and Zhao in \cite{PZ} on the unit ball of $\C^n$. The proof relies on the properties of two closely related operators. The first one is a Toeplitz-like operator $T^\beta_\mu$ (see \eqref{eq:2.T}), depending on a parameter $\beta\in\N^*$ and on a finite positive Borel measure $\mu$, and the main issue consists in identifying functional spaces that can act as domain and/or codomain of such an operator. The second operator, $S^{s,r}_{t,\mu}$ (see \eqref{eq:2.S}), depends on~$\mu$ and three positive real parameters $r$, $s$, $t>0$, and its norm can be used to bound the norm of the operators $T^\beta_\mu$, under suitable assumptions. In particular, the key step in the proof of the necessity implication in the case $0<\lambda<1$ consists in finding criteria for a measure to be $(\lambda,\gamma)$-skew Carleson. These criteria are expressed in terms of mapping properties of the two operators $T^\beta_\mu$ and $S^{s,r}_{t,\mu}$ in the technical Propositions~\ref{th:TtoCarl} and \ref{lemmatretre}.

\smallskip
The paper is structured as follows. In Section~2 we shall collect the preliminary results and definitions. In Section~3 we shall study the properties of the operators $T^\beta_\mu$ and $S^{s,r}_{t,\mu}$ and prove our main result.

\section{Preliminary results}

In this section we collect the precise definitions and preliminary results we shall need in the rest of the paper. 

From now on, $D\subset\subset\C^n$ will be a bounded strongly pseudoconvex domain in $\C^n$ with smooth $C^\infty$ boundary. Furthermore, we shall
use the following notations:
\begin{itemize}
\item $\delta\colon D\to\R^+$ will denote the Euclidean distance from the boundary of $D$,
that is $\delta(z)=d(z,\de D)$;
\item given two non-negative functions $f$, ~$g\colon D\to\R^+$ we shall write $f\preceq g$
to say that there is $C>0$ such that $f(z)\le C g(z)$ for all $z\in D$. The constant $C$ is 
independent of~$z\in D$, but it might depend on other parameters ($r$, $\theta$, etc.);
\item given two strictly positive functions $f$, ~$g\colon D\to\R^+$ we shall write $f\approx g$
if $f\preceq g$ and $g\preceq f$, that is
if there is $C>0$ such that $C^{-1} g(z)\le f(z)\le C g(z)$ for all $z\in D$;
\item $\nu$ will be the Lebesgue measure;
\item $\mathcal{O}(D)$ will denote the space of holomorphic functions on~$D$, endowed with the topology of uniform convergence on compact subsets;
\item given $0< p< +\infty$, the \emph{Bergman space} $A^p(D)$ is the (Banach if $p\ge 1$) space
$L^p(D)\cap\mathcal{O}(D)$, endowed with the $L^p$-norm;
\item more generally, if $\mu$ is a positive finite Borel measure on~$D$ and $0<p<+\infty$ we shall denote by $L^p(\mu)$ the set of complex-valued $\mu$-measurable functions $f\colon D\to\C$ such that
\[
\|f\|_{p,\mu}:=\left[\int_D |f(z)|^p\,d\mu(z)\right]^{1/p}<+\infty\;.
\]
If $\mu=\delta^\beta\nu$ for some $\beta\in\R$, we shall denote by $A^p(D,\beta)$ 
the \emph{weighted Bergman space}
\[
A^p(D,\beta)=L^p(\delta^\beta \nu)\cap\mathcal{O}(D)\;,
\]
and we shall write $\|\cdot\|_{p,\beta}$ instead of $\|\cdot\|_{p,\delta^\beta\nu}$;
\item $K\colon D\times D\to\C$ will be the Bergman kernel of~$D$;
\item for each $z_0\in D$ we shall denote by $k_{z_0}\colon D\to\C$ the \emph{normalized
Bergman kernel} defined by 
\[
k_{z_0}(z)=\frac{K(z,z_0)}{\sqrt{K(z_0,z_0)}}=\frac{K(z,z_0)}{\|K(\cdot,z_0)\|_2}\;;
\]
\item given $r\in(0,1)$ and $z_0\in D$, we shall denote by $B_D(z_0,r)$ the Kobayashi ball
of center~$z_0$ and radius $\frac{1}{2}\log\frac{1+r}{1-r}$.
\end{itemize}

\noindent We refer to, e.g., \cite{A,A1,JP,K}, for definitions, basic properties and applications to geometric function theory of the Kobayashi distance; and to \cite{Ho,Ho1,Kr,R} for definitions and basic properties of the Bergman kernel. 

Let us now recall a number of results we shall need on the Kobayashi geometry of strongly pseudoconvex domains. 

\begin{lemma}[\textbf{\cite[Corollary 7]{Li}, \cite[Lemma 2.1]{AS}}]
\label{sei} 
Let $D\subset\subset\C ^n$ be a bounded  strongly pseudoconvex domain, and $r\in(0,1)$. Then 
\[
\nu\bigl(B_D(\cdot,r)\bigr)\approx \delta^{n+1}\;,
\]
(where the constant depends on~$r$).
\end{lemma}

\begin{lemma}[\textbf{\cite[Lemma 2.2]{AS}}]
\label{sette} 
Let $D\subset\subset\C ^n$ be a bounded  strongly pseudoconvex domain. Then there is $C>0$ such that
\[
 \frac{1-r}{C}\delta(z_0)\le  \delta(z) \le \frac{C}{1-r}\delta(z_0)
\]
for all $r\in(0,1)$, $z_0\in D$ and $z\in B_D(z_0,r)$.
\end{lemma}

We shall also need the existence of suitable coverings by Kobayashi balls:

\begin{definition}
\label{def:lattice}
Let $D\subset\subset\C^n$ be a bounded domain, and $r>0$. A \emph{$r$-lattice} in~$D$
is a sequence $\{a_k\}\subset D$ such that $D=\bigcup_{k} B_D(a_k,r)$ and 
there exists $m>0$ such that any point in~$D$ belongs to at most $m$ balls of the
form~$B_D(a_k,R)$, where $R=\frac{1}{2}(1+r)$.
\end{definition}

The existence of $r$-lattices in bounded strongly pseudoconvex domains is ensured by
the following result: 

\begin{lemma}[\textbf{\cite[Lemma 2.5]{AS}}]
\label{uno}
Let $D\subset\subset\C ^n$ be a bounded strongly pseudoconvex domain. Then for every $r\in(0,1)$ there exists an $r$-lattice in~$D$, that is there exist
 $m\in{\bf N}$ and a sequence 
$\{a_k\}\subset D$ of points such that
$D=\bigcup_{k=0}^\infty B_D(a_k,r)$
and no point of $D$ belongs to more than $m$ of the balls $B_D(a_k,R)$,
where $R={\frac12}(1+r)$.
\end{lemma}

We shall use a submean estimate for nonnegative plurisubharmonic functions on Kobayashi balls:

\begin{lemma}[\textbf{\cite[Corollary 2.8]{AS}}] 
 \label{due}
 Let $D\subset\subset\C ^n$ be a bounded strongly pseudoconvex domain. Given $r\in(0,1)$, set $R={\frac12}(1+r)\in(0,1)$. Then there exists a constant $K_r>0$ depending on~$r$ such that
\[
\forall{z_0\in D\;\forall z\in B_D(z_0,r)}\ \ \ \ \chi(z)\le 
{\frac{K_r}{\nu\left(B_D(z_0,r)\right)}}\int_{B_D(z_0,R)}\chi\,d\nu
\]
for every nonnegative plurisubharmonic function $\chi\colon D\to\R^+$.
\end{lemma}

We shall also need a few estimates on the behavior of the Bergman kernel. The first one is classical (see, e.g., \cite{Ho1}):

\begin{lemma}
\label{BKbasic}
Let $D\subset\subset\C ^n$ be a bounded strongly
pseudoconvex domain. Then
\[
\|K(\cdot,z_0)\|_2=\sqrt{K(z_0,z_0)}\approx \delta(z_0)^{-(n+1)/2}\qquad
\hbox{and}\qquad \|k_{z_0}\|_2\equiv 1
\]
for all $z_0\in D$.
\end{lemma}

A similar estimate but with constants uniform on Kobayashi balls is the following:

\begin{lemma}[\textbf{\cite[Theorem~12]{Li}, \cite[Lemma~3.2 and Corollary 3.3]{AS}}]  
\label{piu}
Let $D\subset\subset\C ^n$ be a bounded strongly
pseudoconvex domain. Then for every $r\in(0,1)$ there exist $c_r>0$ and 
$\delta_r>0$ such that 
if $z_0\in D$ satisfies $\delta(z_0)<\delta_r$ then
\[
\frac{c_r}{\delta(z_0)^{n+1}}\le |K(z,z_0)|\le \frac{1}{c_r\delta(z_0)^{n+1}}
\]
and
\[
\frac{c_r}{\delta(z_0)^{n+1}}\le |k_{z_0}(z)|^2\le \frac{1}{c_r\delta(z_0)^{n+1}}
\]
for all $z\in B_D(z_0,r)$.  
\end{lemma}

\begin{remark}
Note that in the previous lemma the estimates from above hold even when $\delta(z_0)\ge\delta_r$, possibly with a different constant $c_r$. Indeed, when $\delta(z_0)\ge\delta_r$ and $z\in B_D(z_0,r)$ by Lemma~\ref{sette} there is $\tilde\delta_r>0$ such that $\delta(z)\ge \tilde\delta_r$; as a consequence we can find $M_r>0$ such that $|K(z,z_0)|\le M_r$ as soon as $\delta(z_0)\ge\delta_r$ and $z\in B_D(z_0,r)$, and the assertion follows from the fact that $D$ is a bounded domain.
\end{remark}

A very useful integral estimate is the following:

\begin{proposition}[\textbf{\cite[Corollary~11, Theorem~13]{Li}, \cite[Theorem~2.7]{ARS}}]
\label{intest}
Let $D\subset\subset\C ^n$ be a bounded strongly
pseudoconvex domain, and $z_0\in D$. Let $0<p<+\infty$ and $-1<\beta<(n+1)(p-1)$. 
Then 
\[
\int_D |K(z,w)|^p\delta(w)^\beta\,d\nu(w)\preceq \delta(z)^{\beta-(n+1)(p-1)}
\]
and
\[
\int_D |k_z(w)|^p\delta(w)^\beta\,d\nu(w)\preceq\delta(z)^{\beta-(n+1)(\frac{p}{2}-1)}\;.
\]
\end{proposition}


Finally, the normalized Bergman kernel can be used to build functions belonging to suitable weighted Bergman spaces:

\begin{lemma}[\textbf{\cite[Lemma~2.6]{HLZ}}]
\label{th:fact4}
Let $p>0$ and $\theta>1-\frac{1}{n+1}$ be given, and let $\alpha=(n+1)(\theta-1)>-1$. 
Take $\beta\in\mathbb{N}$ such that $\beta p>\max\{\theta,(p-1)\frac{n}{n+1}+\theta\}$ and put
\[
\tau=(n+1)\left[\frac{\beta}{2}-\frac{\theta}{p}\right]\;.
\]
For each $a\in D$ set $f_a=\delta(a)^\tau k_a^\beta$. Let $\{a_k\}$ be an $r$-lattice and $\mathbf{c}=\{c_k\}\in\ell^p$, and put
\[
f=\sum_{k=0}^\infty c_kf_{a_k}\;.
\]
Then $f\in A^p(D,\alpha)$ with $\|f\|_{p,\alpha}\preceq \|\mathbf{c}\|_p$.
\end{lemma}

We also need to recall a few definitions and results about Carleson measures. 

\begin{definition}
\label{def:Carluno}
Let $0<p$, $q<+\infty$ and $\alpha>-1$. A \emph{$(p,q;\alpha)$-skew Carleson measure} is a
finite positive Borel measure $\mu$ such that 
\[
\int_D |f(z)|^q\,d\mu(z)\preceq \|f\|_{p,\alpha}^q
\]
for all $f\in A^p(D,\alpha)$. In other words, $\mu$ is $(p,q;\alpha)$-skew Carleson if $A^p(D,\alpha)\hookrightarrow L^q(\mu)$ continuously. In this case we shall denote by $\|\mu\|_{p,q;\alpha}$ the operator norm of the inclusion $A^p(D,\alpha)\hookrightarrow L^q(\mu)$.
\end{definition}

\begin{remark}
When $p=q$ we recover the usual (non-skew) notion of Carleson measure for $A^p(D,\alpha)$.
\end{remark} 

\begin{definition}
\label{def:Carldue}
Let $\theta\in\R$, and let $\mu$ be a finite positive Borel measure on~$D$. Given $r\in(0,1)$, let
$\hat\mu_{r,\theta}\colon D\to\R$ be defined by
\[
\hat\mu_{r,\theta}(z)=\frac{\mu\bigl(B_D(z,r)\bigr)}{\nu\bigl(B_D(z,r)\bigr)^\theta}\;;
\]
we shall write $\hat\mu_r$ for $\hat\mu_{r,1}$. 

We say that $\mu$ is a \emph{geometric $\theta$-Carleson measure} if $\hat\mu_{r,\theta}\in L^\infty(D)$ for all $r\in(0,1)$, that is if 
for every $r>0$ we have
\[
\mu\bigl(B_D(z,r)\bigr)\preceq \nu\bigl(B_D(z,r)\bigr)^\theta
\] 
for all $z\in D$, where the constant depends only on~$r$.
\end{definition}

Notice that Lemma~\ref{sei} yields
\begin{equation}
\hat\mu_{r,\theta}\approx \delta^{-(n+1)(\theta-1)}\hat\mu_r\;.
\label{eq:hat}
\end{equation}

In \cite{ARS} we proved (among other things) that, if $p\ge 1$, a measure $\mu$ is $(p,p;\alpha)$-skew Carleson if and only if it is geometric $\theta$-Carleson, where $\alpha=(n+1)(\theta-1)$. 
Hu, Lv and Zhu in \cite{HLZ} have given a similar geometric characterization of $(p,q;\alpha)$-skew Carleson measures for all values of $p$ and $q$; to state their results we need another definition.

 \begin{definition}
 \label{def:Berez}
 Let $\mu$ be a finite positive Borel measure on~$D$, and $s>0$. The \emph{Berezin transform} of \emph{level}~$s$ of~$\mu$ is the function $B^s\mu\colon D\to\R^+\cup\{+\infty\}$ given by
 \[
 B^s\mu(z)=\int_D |k_z(w)|^s\,d\mu(w)\;.
 \]
 \end{definition}

The geometric characterization of $(p,q;\alpha)$-skew Carleson measures is different according to whether $p\le q$ or $p>q$. We first state the characterization for the case $p\le q$.

\begin{theorem}
\label{carthetaCarluno}
Let $D\subset\subset\C ^n$ be a bounded strongly
pseudoconvex domain. Let $0< p\le q<+\infty$ and $1-\frac{1}{n+1}<\theta$; set $\alpha=(n+1)(\theta-1)>-1$.
Then the following assertions are
equivalent:
\begin{itemize}
\item[(i)] $\mu$ is a $(p,q;\alpha)$-skew Carleson measure;
\item[(ii)] $\mu$ is a geometric $\frac{q}{p}\theta$-Carleson measure;
\item[(iii)] there exists $r_0\in(0,1)$ such that $\hat\mu_{r_0,\frac{q}{p}\theta}\in L^\infty(D)$;
\item[(iv)] for every $r\in(0,1)$ and for every $r$-lattice $\{a_k\}$ in $D$ we have
\[
\mu\bigl(B_D(a_k,r)\bigr)\preceq \nu\bigl(B_D(a_k,r)\bigr)^{\frac{q}{p}\theta}\;;
\]
\item[(v)] there exists $r_0\in(0,1)$ and a $r_0$-lattice $\{a_k\}$ in $D$ such that
\[
\mu\bigl(B_D(a_k,r_0)\bigr)\preceq\nu\bigl(B_D(a_k,r_0)\bigr)^{\frac{q}{p}\theta}\;;
\]
\item[(vi)] for some (and hence all) $s>\theta\frac{q}{p}$ we have
\[
B^s\mu(a) \preceq \delta(a)^{(n+1)\left(\theta\frac{q}{p}-\frac{s}{2}\right)}\;;
\]
\item[(vii)] there exists $C>0$ such that for some (and hence all) $t>0$ 
we have
\[
\int_D |K(z,a)|^{\theta\frac{q}{p}+\frac{t}{n+1}}\,d\mu(z)\preceq \delta(a)^{-t}\;.
\]
\end{itemize}
Moreover we have
\begin{equation}\label{norme}
\|\mu\|_{p,q;\alpha}
\approx \|\hat\mu_{r,\frac{q}{p}\theta}\|_\infty\approx\|\delta^{-(n+1)(\frac{q}{p}\theta-1)}\hat\mu_r\|_\infty
\approx \|\delta^{(n+1)\left(\frac{s}{2}-\theta\frac{q}{p}\right)}B^s\mu \|_\infty\;.
\end{equation}
\end{theorem}

\begin{proof} 
The equivalence of (i)--(vi), as well as the equivalence for the norms, follows from \cite[Theorem~3.1]{HLZ} (and the equivalence of (ii)--(v) was already in \cite{ARS}).

Now, by Lemma~\ref{BKbasic}, (vi) is equivalent to
\[
\int_D |K(z,a)|^s\,d\mu(z)\preceq \delta(a)^{(n+1)\left(\theta\frac{q}{p}-s\right)}\;.
\]
Setting $t=(n+1)\left(s-\theta\frac{q}{p}\right)$, which is positive if and only if $s>\theta\frac{q}{p}$,
we see that (vi) is equivalent to
\[
\int_D |K(z,a)|^{\theta\frac{q}{p}+\frac{t}{n+1}}\,d\mu(z)\preceq \delta(a)^{-t}\;,
\] 
that is to (vii).
\end{proof}

The geometric characterization of $(p,q;\alpha)$-skew Carleson measures when $p>q$ has a slightly different flavor: 

\begin{theorem}
\label{carthetaCarldue}
Let $D\subset\subset\C ^n$ be a bounded strongly
pseudoconvex domain. Let $0< q< p<+\infty$ and $1-\frac{1}{n+1}<\theta$; put $\alpha=(n+1)(\theta-1)>-1$.
Then the following assertions are
equivalent:
\begin{itemize}
\item[(i)] $\mu$ is a $(p,q;\alpha)$-skew Carleson measure;
\item[(ii)] $\hat\mu_r \delta^{-\alpha\frac{q}{p}}\in L^{\frac{p}{p-q}}(D)$ for some (and hence any) $r\in(0,1)$;
\item[(iii)] $\hat\mu_{r,\theta}\in L^{\frac{p}{p-q}}(D,\alpha)$ for some (and hence any) $r\in(0,1)$;
\item[(iv)] $\hat\mu_{r,\theta\frac{q}{p}}\in L^{\frac{p}{p-q}}\bigl(D,-(n+1)\bigr)$ for some (and hence any) $r\in(0,1)$;
\item[(v)] for some (and hence any) $r\in(0,1)$ and for some (and hence any) $r$-lattice $\{a_k\}$ in $D$ we have
$\{\hat\mu_{r,\theta\frac{q}{p}}(a_k)\}\in\ell^{\frac{p}{p-q}}$;
\item[(vi)] for some (and hence any) $r\in(0,1)$ and for some (and hence any) $r$-lattice $\{a_k\}$ in $D$ we have
$\{\hat\mu_{r}(a_k)\delta(a_k)^{(n+1)\left(1-\theta\frac{q}{p}\right)}\}\in\ell^{\frac{p}{p-q}}$;
\item[(vii)] for some (and hence all) $s>\theta\frac{q}{p}+\frac{n}{n+1}\left(1-\frac{q}{p}\right)$ we have
\[
\delta^{-(n+1)\left(\theta\frac{q}{p}-\frac{s}{2}\right)}B^s\mu\in L^{\frac{p}{p-q}}\bigl(D,-(n+1)\bigr)\;;
\]
\item[(viii)] for some (and hence all) $s>\theta\frac{q}{p}+\frac{n}{n+1}\left(1-\frac{q}{p}\right)$ we have
\[
\delta^{-(n+1)\left(\theta-\frac{s}{2}\right)}B^s\mu\in L^{\frac{p}{p-q}}(D,\alpha)\;;
\]
\item[(ix)] for some (and hence all) $s>\theta\frac{q}{p}+\frac{n}{n+1}\left(1-\frac{q}{p}\right)$ we have
\[
\delta^{-(n+1)\left(\theta\frac{q}{p}-\frac{s}{2}+\frac{p-q}{p}\right)}B^s\mu\in L^{\frac{p}{p-q}}(D)\;;
\]
\item[(x)] for some (and hence all) $t>(n+1)\left(1-\frac{q}{p}\right)\left(\frac{n}{n+1}-\theta\right)$ we have
\[
\delta^t \int_D |K(\cdot,w)|^{\theta+\frac{t}{n+1}}\,d\mu(w)\in L^{\frac{p}{p-q}}(D,\alpha)\;.
\]
\end{itemize}
Moreover we have
\begin{equation}\label{norme2}
\|\mu\|_{p,q;\alpha}
\approx \|\delta^{-(n+1)(\theta-\frac{s}{2})}B^s\mu\|_{\frac{p}{p-q},\alpha}
\approx  \|\delta^{-(n+1)(\theta-1)\frac{q}{p}}\hat\mu_r \|_{\frac{p}{p-q}}\\
\end{equation}
\end{theorem}

\begin{proof}
The equivalence of (i), (ii), (vi) and (ix), as well as the equivalence of the norms, is in \cite[Theorem~3.3]{HLZ}.

Recalling that, by Lemma~\ref{sei}, $\hat\mu_{r,\theta}\approx \hat\mu_r \delta^{(n+1)(1-\theta)}$, 
it is easy to see that the equalities
\[
\begin{aligned}
-(n+1)(\theta-1)\frac{q}{p}\frac{p}{p-q}&=(n+1)(1-\theta)\frac{p}{p-q}+(n+1)(\theta-1)\\
&=(n+1)\left(1-\theta\frac{q}{p}\right)\frac{p}{p-q}-(n+1)
\end{aligned}
\]
yield the equivalence of (ii), (iii) and (iv).

The fact that $\hat\mu_{r,\theta}\approx \hat\mu_r \delta^{(n+1)(1-\theta)}$ immediately yields
the equivalence between (v) and (vi).


The equalities
\[
\begin{aligned}
-(n+1)\left(\theta\frac{q}{p}-\frac{s}{2}\right)\frac{p}{p-q}-(n+1)&=-(n+1)\left(\theta-\frac{s}{2}\right)\frac{p}{p-q}+(n+1)(\theta-1)\\
&=-(n+1)\left(\theta\frac{q}{p}-\frac{s}{2}+\frac{p-q}{p}\right)\frac{p}{p-q}
\end{aligned}
\]
yield the equivalence of (vii), (viii) and (ix).

Finally, by Lemma~\ref{BKbasic}, (viii) is equivalent to
\[
\delta^{-(n+1)(\theta-s)}\int_D |K(\cdot,w)|^s\,d\mu(w)\in L^{\frac{p}{p-q}}(D,\alpha)\;,
\]
and this is equivalent to (x) 
via the substitution $s=\theta+\frac{t}{n+1}$.
\end{proof}


A consequence of these two theorems is that the property of being $(p,q;\alpha)$-skew Carleson actually depends only on the quotient $q/p$ and on $\alpha$. We shall then introduce the following definition:

\begin{definition}
\label{def:sCdef}
Let $\lambda>0$ and $\gamma>1-\frac{1}{n+1}$. 
A finite positive Borel measure $\mu$ is \emph{$(\lambda,\gamma)$-skew Carleson} if either
\begin{itemize}
\item[--] $\lambda\ge 1$ and $\hat\mu_{r_0,\lambda\gamma}\in L^\infty(D)$ for some (and hence all) $r_0\in(0,1)$; or,
\item[--] $\lambda<1$ and $\hat\mu_{r_0,\gamma}\in L^{\frac{1}{1-\lambda}}\bigl(D,(n+1)(\gamma-1)\bigr)$ for some (and hence all) $r_0\in(0,1)$.
\end{itemize}
\end{definition}

Thus Theorems~\ref{carthetaCarluno} and \ref{carthetaCarldue} say that $\mu$ is $(p,q;\alpha)$-skew Carleson if and only if it is $(q/p,\gamma)$-skew Carleson, where $\alpha=(n+1)(\gamma-1)$.
In particular, we shall write $\|\mu\|_{q/p,\gamma}$ instead of $\|\mu\|_{p,q;(n+1)(\gamma-1)}$.

\smallskip We end this section with the following easy (but useful) consequence of this definition:

\begin{lemma}
\label{th:last}
Let $D\subset\subset\C^n$ be a bounded strongly pseudoconvex domain, $\lambda>0$ and $\gamma>1-\frac{1}{n+1}$. Let $\mu$ be a $(\lambda,\gamma)$-skew Carleson measure, and $\beta>\lambda\left(\frac{n}{n+1}-\gamma\right)$. Then $\mu_\beta=\delta^{(n+1)\beta}\mu$ is a $(\lambda, \gamma+\frac{\beta}{\lambda})$-skew Carleson measure with $\|\mu_\beta\|_{\lambda,\gamma+\frac{\beta}{\lambda}}\approx \|\mu\|_{\lambda,\gamma}$. 
\end{lemma}

\begin{proof}
First of all, remark that using Lemmas~\ref{sei} and~\ref{sette} it is easy to check that
\[
(\widehat{\mu_\beta})_r\approx \delta^{(n+1)\beta}\hat\mu_r\;.
\] 

Assume $0<\lambda<1$. By Theorem~\ref{carthetaCarldue}, we know that $\hat\mu_r\delta^{-(n+1)(\gamma-1)\lambda}\in L^{\frac{1}{1-\lambda}}(D)$. Therefore
\[
(\widehat{\mu_\beta})_r\delta^{-(n+1)(\gamma+\frac{\beta}{\lambda}-1)\lambda}\approx \hat\mu_r\delta^{-(n+1)(\gamma-1)\lambda}\in L^{\frac{1}{1-\lambda}}(D)\;,
\]
and again Theorem~\ref{carthetaCarldue} implies that $\mu_\beta$ is $(\lambda,\gamma+\frac{\beta}{\lambda})$-skew Carleson with $\|\mu_\beta\|_{\lambda,\gamma+\frac{\beta}{\lambda}}\approx \|\mu\|_{\lambda,\gamma}$.

If $\lambda\ge 1$, again Lemmas~\ref{sei} and~\ref{sette} yield 
\[
(\widehat{\mu_\beta})_{r,\lambda\gamma+\beta}\approx (\widehat{\mu_\beta})_r\delta^{-(n+1)(\lambda\gamma+\beta-1)}\approx \hat\mu_r \delta^{-(n+1)(\lambda\gamma-1)}\approx
\hat\mu_{r,\lambda\gamma}
\]
and Theorem~\ref{carthetaCarluno} yields the assertion.
%
\end{proof}

\section{Proof of the main result}

The proof of the main result will use two closely related operators. The first one is a Toeplitz-like operator $T^\beta_\mu$, depending on a parameter $\beta\in\N^*$ and on a finite positive Borel measure $\mu$, defined by the formula
\begin{equation}
T^\beta_\mu f(z)=\int_D K(z,w)^\beta f(w)\,d\mu(w)
\label{eq:2.T}
\end{equation}
for suitable functions $f\colon D\to\C$; part of the work will exactly be identifying functional spaces that can act as domain and/or codomain of such an operator. We need $\beta$ to be a natural number because the Bergman kernel in general might have zeroes and $D$ is not necessarily simply connected.

The second operator $S^{s,r}_{t,\mu}$ depends on~$\mu$ and three positive real parameters $r$, $s$, $t>0$ and is defined by
\begin{equation}
S^{s,r}_{\mu,t}f(z)=\delta(z)^{(n+1)s}\int_D |k_z(w)|^t |f(w)|^r\,d\mu(w)\;,
\label{eq:2.S}
\end{equation}
again for suitable functions $f\colon D\to\C$. This time the exponents do not need to be integers. Notice that Lemma~\ref{BKbasic} yields
\begin{equation}
|S^{s,r}_{\mu,t}f(z)|\approx \delta(z)^{(n+1)(s+\frac{t}{2})}\int_D |K(z,w)|^t|f(w)|^r\,d\mu(w)\;.
\label{eq:biS}
\end{equation}

Therefore it is not surprising that, under suitable hypotheses we can use the norm of the operators $S^{s,r}_{t,\mu}$ to bound the norm of the operators $T^\beta_\mu$. We start with a preliminary lemma:

\begin{lemma}
\label{th:frombetatot}
Let $D\subset\subset\C ^n$ be a bounded strongly pseudoconvex domain, and $\mu$ a positive finite Borel measure on $D$. Then for every $\beta\ge t>0$ we have
\[
\int_D |K(z,w)|^\beta\,d\mu(w)\preceq \delta(z)^{-(n+1)(\beta-t)}\int_D |K(z,w)|^t\,d\mu(w)\;.
\]
\end{lemma}

\begin{proof}
Given $z\in D$ put $D_1=\{w\in D\mid |K(z,w)|\ge 1\}$ and $D_0=D\setminus D_1$. Then
\[
\begin{aligned}
\int_D |K(z,w)|^\beta\,d\mu(w) &=\int_{D_0} |K(z,w)|^t |K(z,w)|^{\beta-t}\,d\mu(w)+\int_{D_1}|K(z,w)|^t |K(z,w)|^{\beta-t}\,d\mu(w)\\
&\le \int_{D_0} |K(z,w)|^t \,d\mu(w)+\sup_{w\in D}|K(z,w)|^{\beta-t}\int_{D_1}|K(z,w)|^t\,d\mu(w)\\
&\preceq \sup_{w\in D}|K(z,w)|^{\beta-t}\int_{D}|K(z,w)|^t\,d\mu(w)\;,
\end{aligned}
\]
and the assertion follows from the known estimate
\[
\sup_{w\in D}|K(z,w)|\preceq \delta(z)^{-(n+1)}\;.
\]
\end{proof}

We then have the following estimates.

\begin{lemma}
\label{th:fromStoTnew}
Let $D\subset\subset\C ^n$ be a bounded strongly
pseudoconvex domain, and $\mu$ a positive finite Borel measure on $D$. Choose $r\ge 1$, $s$, $t$, $p$, $q>0$, $\sigma$, $\theta_1>1-\frac{1}{n+1}$ and $\beta\in\N^*$. Then:
\begin{itemize}
\item[(i)] if $r=1$, $q\ge p$, $\beta\ge t$ and $\theta_1\le q\left[\frac{\sigma}{p}+\frac{1}{q}-\frac{1}{p}+\frac{t}{2}-\beta-s\right]$ we have
\[
\|T^\beta_\mu f\|_{p,(n+1)(\sigma-1)}\preceq \|S^{s,1}_{\mu,t}f\|_{q,(n+1)(\theta_1-1)}\;;
\]
\item[(ii)] if $r>1$, $q\ge p/r$, $\beta\ge t/r$,  we have
\[
\|T^\beta_\mu f\|_{p,(n+1)(\sigma-1)}\preceq \|\delta^{-(n+1)(\gamma-\frac{\alpha}{2})} B^\alpha\mu\|^{1/r'}_{\frac{1}{1-\lambda},(n+1)(\gamma-1)}\|S^{s,r}_{\mu,t}f\|^{1/r}_{q,(n+1)(\theta_1-1)}\;,
\]
where $r'$ is the conjugate exponent of~$r$ and
\[
\lambda=1+r'\left[\frac{1}{qr}-\frac{1}{p}\right]<1\;, \quad \alpha=r'\left(\beta-\frac{t}{r}\right)\;,\quad \gamma=\frac{r'}{\lambda}\left[\beta+\frac{1}{r}\left(s-\frac{t}{2}\right)+
\frac{\theta_1}{qr}-\frac{\sigma}{p}\right]\;.
\]
\end{itemize}
\end{lemma}

\begin{proof}
(i) Lemma~\ref{th:frombetatot}, applied to the measure $|f|\mu$, and \eqref{eq:biS}  yield
\[
\begin{aligned}
|T^\beta_\mu f(z)|^p&\le\left[\int_D |K(z,w)|^\beta |f(w)|\,d\mu(w)\right]^p\\
&\preceq \delta(z)^{-(n+1)(\beta-t)p}\left[\int_D |K(z,w)|^t |f(w)|\,d\mu(w)\right]^p\\
&\preceq \delta(z)^{-(n+1)(\beta-\frac{t}{2}+s)p}|S^{s,1}_{\mu,t}f(z)|^p\;.
\end{aligned}
\]
Therefore using H\"older's inequality we obtain
\[
\begin{aligned}
\|T^\beta_\mu f\|^p_{p,(n+1)(\sigma-1)}&\preceq \int_D |S^{s,1}_{\mu,t}f(z)|^p \delta(z)^{-(n+1)[(\beta-\frac{t}{2}+s)p+1-\sigma]}\,d\nu(z)\\
&\preceq\left[\int_D |S^{s,1}_{\mu,t}f(z)|^q \delta(z)^{-(n+1)[(\beta-\frac{t}{2}+s)q+\frac{(1-\sigma)q}{p}]}\,d\nu(z)\right]^{p/q}\\
&=\|S^{s,1}_{\mu,t}f\|^p_{q,(n+1)q[\frac{\sigma-1}{p}+\frac{t}{2}-\beta-s]}\\
&\preceq \|S^{s,1}_{\mu,t}f\|^p_{q,(n+1)(\theta_1-1)}\;,
\end{aligned}
\]
where the last step follows from \cite[Lemma~2.10]{ARS}.

(ii) Writing $\beta=\frac{t}{r}+\frac{\alpha}{r'}$, using H\"older's inequality and recalling the definition of the Berezin transform we obtain 
\[
\begin{aligned}
|T^\beta_\mu f(z)|^p&\le\left[\int_D |K(z,w)|^t|f(w)|^r\,d\mu(w)\right]^{p/r}\left[\int_D |K(z,w)|^{\alpha}\,d\mu(w)\right]^{p/r'}\\
&\preceq \left[\int_D |K(z,w)|^t|f(w)|^r\,d\mu(w)\right]^{p/r} \delta(z)^{-\frac{(n+1)\alpha p}{2r'}}|B^{\alpha}\mu(z)|^{p/r'}\;.
\end{aligned}
\]
Therefore, recalling that $\alpha/r'=\beta-t/r$ and using again H\"older's inequality, we have
\[
\begin{aligned}
\|T^\beta_\mu f\|^p_{p,(n+1)(\sigma-1)}&\preceq \int_D \left[\int_D |K(z,w)|^t|f(w)|^r\,d\mu(w)\right]^{p/r}|B^\alpha\mu(z)|^{p/r'}\delta(z)^{(n+1)(\sigma-1-\frac{\alpha p}{2r'})}
\,d\nu(z)\\
&\preceq \int_D |S^{s,r}_{\mu, t}f(z)|^{p/r}|B^\alpha \mu(z)|^{p/r'}\delta(z)^{(n+1)p[\frac{\sigma-1}{p}-\frac{\alpha}{2r'}-(s+\frac{t}{2})\frac{1}{r}]}\,d\nu(z)\\
&\le\left[\int_D |S^{s,r}_{\mu,t}f(z)|^q\delta(z)^{(n+1)(\theta_1-1)}\,d\nu(z)\right]^{p/qr}\left[\int_D |B^\alpha\mu(z)|^{\frac{pqr}{r'(qr-p)}}\delta(z)^{(n+1)(\tau-1)}\,d\nu(z)\right]^{1-\frac{p}{qr}}\\
&=\|S^{s,r}_{\mu,t}f\|^{p/r}_{q,(n+1)(\theta_1-1)}\|\delta^{-(n+1)(\gamma-\frac{\alpha}{2})}B^\alpha\mu\|^{p/r'}_{\frac{1}{1-\lambda},(n+1)(\gamma-1)}\;,
\end{aligned}
\]
where $\lambda$ and $\gamma$ are as in the statement and
\[
\tau=\frac{r'}{1-\lambda}\left[\frac{\sigma}{p}-\frac{\theta_1}{qr}-\frac{\alpha}{2r'}-\left(s+\frac{t}{2}\right)\frac{1}{r}\right]\;.
\]
\end{proof}

\begin{corollary}
\label{th:fromStoTfin}
Let $D\subset\subset\C ^n$ be a bounded strongly
pseudoconvex domain, and $\mu$ a positive finite Borel measure on $D$. For $r>1$, $s$, $t>0$, $\tilde p$, $q>0$, $\alpha>0$, $\gamma\in\R$ and $\theta_1>1-\frac{1}{n+1}$
assume that 
\[
\beta=\frac{t}{r}+\frac{\alpha}{r'}\in\N\qquad \mathit{and}\qquad \lambda=1+\frac{r}{\tilde p}-\frac{1}{q}<1\;,
\]
where $r'=r/(r-1)$ is the conjugate exponent of $r$.
Then 
\[
\|T^\beta_\mu f\|_{\tau,(n+1)(\sigma-1)}\preceq \|\delta^{-(n+1)(\gamma-\frac{\alpha}{2})}B^\alpha\mu\|^{1/r'}_{\frac{1}{1-\lambda},(n+1)(\gamma-1)}\|S^{s,r}_{\mu,t}f\|^{1/r}_{q,(n+1)(\theta_1-1)}\;,
\]
where 
\[
\tau=\frac{1}{1-\lambda+\frac{1}{\tilde p}}
\]
and if $\sigma>1-\frac{1}{n+1}$, we have
\begin{equation}
\sigma
=
\tau\left[\frac{\alpha-\lambda\gamma}{r'} + \frac{1}{q r}\left(\theta_1+ q\left(s+\frac{t}{2}\right)\right)\right]\;.
\label{eq:sigma}
\end{equation}
\end{corollary}

%
\begin{proof}
%
The assertion is a consequence of Lemma~\ref{th:fromStoTnew}.(ii) applied with $p=\tau$. Indeed,
first of all, since $\lambda<1$, we have $\tilde p>rq$; from this it follows that
\[
\frac{1-r}{\tilde p}>\frac{1-r}{qr}\quad\Longleftrightarrow\quad 1-\lambda+\frac{1}{\tilde p}>\frac{1}{qr} \quad \Longleftrightarrow\quad q>\frac{\tau}{r}
\]
as needed. Furthermore
\[
1+r'\left[\frac{1}{qr}-\frac{1}{\tau}\right]=1+\frac{r'-1}{q}-r'+r'+\frac{rr'}{\tilde p}-\frac{r'}{q}-\frac{r'}{\tilde p}=1+\frac{r}{\tilde p}-\frac{1}{q}=\lambda
\]
and
\[
\frac{r'}{\lambda}\left[\beta+\frac{1}{r}\left(s-\frac{t}{2}\right)+\frac{\theta_1}{qr}-\frac{\sigma}{\tau}\right]=\frac{r'}{\lambda}\left[\beta+\frac{1}{r}\left(s-\frac{t}{2}\right)+\frac{\theta_1}{qr}+\frac{\lambda\gamma-\alpha}{r'}-\frac{1}{qr}\left(\theta_1+q\left(s+\frac{t}{2}\right)\right)\right]=\gamma\;.
\]
\end{proof}

The mapping properties of the operators $T^\beta_\mu$ and $S^{s,r}_{t,\mu}$ can be used to give criteria for a measure $\mu$ to be $(\lambda,\gamma)$-skew Carleson, which is particularly useful when $\lambda<1$. We start with $T^\beta_\mu$:

\begin{proposition}
\label{th:TtoCarl}
Let $D\subset\subset\C^n$ be a bounded strongly pseudoconvex domain, and $\mu$ a positive finite Borel measure on $D$. Take $0<q<p<\infty$,
$\theta_1$, $\theta_2>1-\frac{n}{n+1}$ and $\beta\in\N$ such that
\[
\beta>\frac{1}{p}\max\{1,\theta_1,p-1+\theta_1\}\;.
\]
Put
\[
\lambda=1+\frac{1}{p}-\frac{1}{q}<1\qquad\mathit{and}\qquad \gamma=\frac{1}{\lambda}\left(\beta+\frac{\theta_1}{p}-\frac{\theta_2}{q}\right)\;.
\]
Assume that $T^\beta_\mu$ is bounded from $A^{p}\bigl(D,(n+1)(\theta_1-1)\bigr)$ to $A^{q}\bigl(D,(n+1)(\theta_2-1)\bigr)$, with operator norm $\|T^\beta_\mu\|$. Then $\mu$ is $(\lambda,\gamma)$-skew Carleson, and
\[
\|\delta^{-(n+1)(\gamma-\frac{\alpha}{2})}B^\alpha\mu\|_{\frac{1}{1-\lambda},(n+1)(\gamma-1)}\preceq \|T^\beta_\mu\|
\]
for all $\alpha>\lambda\gamma+\frac{n}{n+1}(1-\lambda)$.
\end{proposition}

\begin{proof}
Let $\{a_k\}$ be an $r$-lattice in $D$, and $\{r_k\}$ a sequence of Rademacher functions (see \cite[Appendix A]{Duren}). 
Set
\[
\tau=(n+1)\left[\frac{\beta}{2}-\frac{\theta_1}{p}\right]\;,
\]
and, for every $a\in D$, put $f_a=\delta(a)^\tau k_a^\beta$. Then Lemma~\ref{th:fact4} implies that
\[
f_t=\sum_{k=0}^\infty c_kr_k(t)f_{a_k}
\]
belongs to $A^{p}\bigl(D,(n+1)(\theta_1-1)\bigr)$ for all $\mathbf{c}=\{c_k\}\in\ell^{p}$, and $\|f\|_{p,(n+1)(\theta_1-1)}\preceq\|\mathbf{c}\|_{p}$.

Since, by assumption, $T^\beta_\mu$ is bounded from $A^{p}\bigl(D,(n+1)(\theta_1-1)\bigr)$ to $A^{q}\bigl(D,(n+1)(\theta_2-1)\bigr)$ we have
\begin{equation*}\begin{aligned}
\|T^\beta_\mu f_t\|^{q}_{q,(n+1)(\theta_2-1)}
&=\int_D\left|\sum_{k=0}^\infty c_k r_k(t) T^\beta_\mu f_{a_k}(z)\right|^{q}\delta(z)^{(n+1)(\theta_2-1)}\,d\nu(z)\\
&\le\|T^\beta_\mu\|^{q}\|f_t\|^{q}_{p,(n+1)(\theta_1-1)}\preceq \|T^\beta_\mu\|^{q}\|\mathbf{c}\|^{q}_{p}\;.
\end{aligned}\end{equation*}
Integrating both sides on $[0,1]$ with respect to $t$ and using Khinchine's inequality (see, e.g., \cite{Luecking}) we obtain
\[
\int_D\left(\sum_{k=0}^\infty |c_k|^2|T^\beta_\mu f_{a_k}(z)|^2\right)^{q/2}\delta(z)^{(n+1)(\theta_2-1)}\,d\nu(z)\preceq \|T^\beta_\mu\|^{q}\|\mathbf{c}\|^{q}_{p}\;.
\]
Set $B_k=B_D(a_k,r)$. 
We have to consider two cases: $q\ge 2$ and $0<q<2$. 

If $q\ge 2$, using the fact that $\|\mathbf{a}\|_{q/2}\le\|\mathbf{a}\|_1$ for every $\mathbf{a}\in\ell^1$ we get
\begin{equation*}\begin{aligned}
\sum_{k=0}^\infty |c_k|^{q}\int_{B_k}|&T^\beta_\mu f_{a_k}(z)|^{q}\delta(z)^{(n+1)(\theta_2-1)}\,d\nu(z)\\
&\le
\int_D\left(\sum_{k=0}^\infty |c_k|^2|T^\beta_\mu f_{a_k}(z)|^2 \chi_{B_k}(z)\right)^{q/2}\delta(z)^{(n+1)(\theta_2-1)}\,d\nu(z)\\
&\le \int_D\left(\sum_{k=0}^\infty |c_k|^2|T^\beta_\mu f_{a_k}(z)|^2\right)^{q/2}\delta(z)^{(n+1)(\theta_2-1)}\,d\nu(z)\;.
\end{aligned}\end{equation*}
If instead $0<q<2$, using H\"older's inequality, we obtain
\begin{equation*}
\begin{aligned}
\sum_{k=0}^\infty |c_k|^{q}\int_{B_k}&|T^\beta_\mu f_{a_k}(z)|^{q}\delta(z)^{(n+1)(\theta_2-1)}\,d\nu(z)\\
&\le
\int_D\left(\sum_{k=0}^\infty |c_k|^2|T^\beta_\mu f_{a_k}(z)|^2\right)^{\frac{q}{2}}\left(\sum_{k=0}^\infty\chi_{B_k}(z)\right)^{1-\frac{q}{2}}\delta(z)^{(n+1)(\theta_2-1)}\,d\nu(z)\\
&\preceq\int_D\left(\sum_{k=0}^\infty |c_k|^2|T^\beta_\mu f_{a_k}(z)|^2\right)^{q/2}\delta(z)^{(n+1)(\theta_2-1)}\,d\nu(z)\;,
\end{aligned}
\end{equation*}
where we used the fact that each $z\in D$ belongs to no more than $m$ of the $B_k$. 

Summing up, for any $q>0$ we have
\[
\sum_{k=0}^\infty |c_k|^{q}\int_{B_k}|T^\beta_\mu f_{a_k}(z)|^{q}\delta(z)^{(n+1)(\theta_2-1)}\,d\nu(z)\preceq \|T^\beta_\mu\|^{q}\|\mathbf{c}\|^{q}_{p}\;.
\]
Now Lemmas~\ref{sei}, \ref{sette} and~\ref{due} (see also \cite[Corollary~2.7]{AS}) yield
\[
|T^\beta_\mu f_{a_k}(a_k)|^{q}\preceq \delta(a_k)^{-(n+1)\theta_2} \int_{B_k} |T^\beta_\mu f_{a_k}(z)|^{q}\delta(z)^{(n+1)(\theta_2-1)}\,d\nu(z)\;,
\]
and so we get
\[
\sum_{k=0}^\infty |c_k|^{q}\delta(a_k)^{(n+1)\theta_2}|T^\beta_\mu f_{a_k}(a_k)|^{q}\preceq \|T^\beta_\mu\|^{q}\|\mathbf{c}\|^{q}_{p}\;.
\]
On the other hand, using Lemmas~\ref{BKbasic} and~\ref{piu}, we obtain
\begin{equation*}\begin{aligned}
T^\beta_\mu f_{a_k}(a_k)
&=\delta(a_k)^\tau \int_D K(a_k,w)^\beta k_{a_k}(w)^\beta\,d\mu(w)\succeq\delta(a_k)^{\tau+(n+1)\frac{\beta}{2}}\int_D |K(a_k,w)|^{2\beta} \,d\mu(w)\\
&\ge \delta(a_k)^{\tau+(n+1)\frac{\beta}{2}}\int_{B_D(a_k,r)} |K(a_k,w)|^{2\beta}\,d\mu(w)\\
&\succeq \delta(a_k)^{\tau-(n+1)\frac{3\beta}{2}}\mu\bigl(B_D(a_k,r)\bigr)=\delta(a_k)^{-(n+1)[\beta+\frac{\theta_1}{p}]}\mu\bigl(B_D(a_k,r)\bigr)\;.
\end{aligned}\end{equation*}
Putting all together we get
\[
\sum_{k=0}^\infty |c_k|^{q}\left(\frac{\mu\bigl(B_D(a_k,r)\bigr)}{\delta(a_k)^{(n+1)\lambda\gamma}}\right)^{q}\preceq \|T^\beta_\mu\|^{q}\|\mathbf{c}\|^{q}_{p}\;.
\]
Set $\mathbf{d}=\{d_k\}$, where
\[
d_k=\frac{\mu\bigl(B_D(a_k,r)\bigr)}{\delta(a_k)^{(n+1)\lambda\gamma}}\;.
\]
Then by duality we get $\{d_k^{q}\}\in\ell^{p/(p-q)}$ with $\|\{d_k^{q}\}\|_{p/(p-q)}\preceq \|T^\beta_\mu\|^{q}$, because $p/(p-q)$ is the conjugate exponent of $p/q>1$. This means that $\mathbf{d}\in\ell^{pq/(p-q)}=\ell^{1/(1-\lambda)}$ with 
\[
\|\mathbf{d}\|_{\frac{1}{1-\lambda}}\preceq \|T^\beta_\mu\|\;.
\]
Since
\[
d_k \approx \hat\mu_r(a_k)\delta(a_k)^{(n+1)(1-\lambda\gamma)}\;,
\]
the assertion then follows from Theorem~\ref{carthetaCarldue}.
\end{proof}

\begin{remark}
Note that a similar result holds also for $\lambda\ge 1$ and can be strengthened to give yet another characterization of skew Carleson measures. Since such result is not needed in the present paper, we prefer to omit it here, and to present it in a forthcoming paper.
\end{remark}

We can now prove a technical result involving the operators $S^{s,r}_{\mu,t}$ that will be crucial for the proof of our main theorem.

\begin{proposition}
\label{lemmatretre}
Let $D\subset\subset\C ^n$ be a bounded strongly
pseudoconvex domain, and $\mu$ a positive finite Borel measure on $D$. Fix $q>1$, $p>0$, $\theta_1$, $\theta_2>1-\frac{1}{n+1}$ and $r,s>0$, and $t>\frac{1}{p}\max\left\{1,\theta_2,p-1+\theta_2\right\}>0$. Set
\[
\lambda=1+\frac{r}{p}-\frac{1}{q}
\qquad\mathrm{and}\qquad 
\gamma=\frac{1}{\lambda}
\left(\frac{t}{2}+\frac{\theta_2 r}{p}-\frac{\theta_1}{q}-s\right)\;.
\] 
Assume that $\lambda>0$ and $\gamma>1-\frac{1}{n+1}$, and that there exists $K>0$ such that 
\begin{equation}
\|S^{s,r}_{\mu,t}f\|_{q,(n+1)(\theta_1-1)}\le K\|f\|^r_{p,(n+1)(\theta_2-1)}
\label{eqltretreu}
\end{equation}
for all $f\in A^p\bigl(D, (n+1)(\theta_2-1)\bigr)$. Then $\mu$ is a $(\lambda,\gamma)$-skew Carleson measure with $\|\mu\|_{\lambda,\mu}\preceq K$.
\end{proposition}

\begin{proof}
Let us first consider the case $\lambda\ge 1$. Given $a\in D$ and $\sigma\in\N$ such that
\[
p\sigma>\theta_2,
\]
set
\[
f_a^\sigma(z)=k_a(z)^\sigma\;,
\]
for $z\in D$. 
By Proposition~\ref{intest} we have
\begin{equation}
\|f_a^\sigma\|^r_{p,(n+1)(\theta_2-1)}\preceq \delta(a)^{(n+1)(\theta_2\frac{r}{p}-\frac{r\sigma}{2})}\;.
\label{eq:mfs}
\end{equation}
Now fix $\rho>0$. Clearly, there is a $\hat\rho>0$ depending only on~$\rho$ such that $z$, $w\in B_D(a,\rho)$ implies $w\in B_D(z,\hat\rho)$ for all $a\in D$. 
By Lemma~\ref{sette} we can find $\delta_1>0$ such that if $\delta(a)<\delta_1$ then $\delta(z)<\delta_{\hat\rho}$ for all $z\in B_D(a,\rho)$, where $\delta_{\hat\rho}>0$ is given by Lemma~\ref{piu}. Then if $\delta(a)<\delta_1$ using Lemmas~\ref{sei}, \ref{sette} and \ref{piu} we have
\begin{equation*}\begin{aligned}
\|S^{s,r}_{\mu,t}f^\sigma_a\|^q_{q,(n+1)(\theta_1-1)}
&=\int_D |S^{s,r}_{\mu,t}f^\sigma_a(z)|^q\delta(z)^{(n+1)(\theta_1-1)}\, d\nu(z)\\
&\ge\int_{B_D(a,\rho)} |S^{s,r}_{\mu,t}f^\sigma_a(z)|^q\delta(z)^{(n+1)(\theta_1-1)}\, d\nu(z)\\
&\succeq\delta(a)^{(n+1)(\theta_1-1)}\int_{B_D(a,\rho)}\delta(z)^{(n+1)qs}\left[\int_D |k_z(w)|^t|f_a^\sigma(w)|^r\,d\mu(w)\right]^q d\nu(z)\\
&\succeq\delta(a)^{(n+1)(\theta_1-1+qs)}\int_{B_D(a,\rho)}\left[\int_{B_D(a,\rho)}
|k_z(w)|^t|f_a^\sigma(w)|^r\,d\mu(w)\right]^q d\nu(z)\\
&\succeq\delta(a)^{(n+1)(\theta_1-1+qs-\frac{1}{2}\sigma rq)}\int_{B_D(a,\rho)}\delta(z)^{\frac{n+1}{2}tq}\left[\int_{B_D(a,\rho)}|K(z,w)|^t\,d\mu(w)\right]^q d\nu(z)\\
&\succeq\delta(a)^{(n+1)(\theta_1-1+qs-\frac{1}{2}\sigma rq+\frac{1}{2}tq)}
\int_{B_D(a,\rho)}\delta(z)^{-(n+1)tq}\mu\bigl(B_D(a,\rho)\bigr)^q\,d\nu(z)\\
&\succeq\delta(a)^{(n+1)(\theta_1+qs-\frac{1}{2}\sigma rq-\frac{1}{2}tq)}\mu\bigl(B_D(a,\rho)\bigr)^q\;.
\end{aligned}\end{equation*}
Recalling \eqref{eqltretreu} and \eqref{eq:mfs}, we get
\[
\begin{aligned}
\mu\bigl(B_D(a,\rho)\bigr) 
&\preceq K\delta(a)^{(n+1)(\frac{\sigma r}{2}+\frac{t}{2}-s-\frac{\theta_1}{q})}\|f_a^\sigma\|_{p,(n+1)(\theta_2-1)}^r\\
&\preceq K\delta(a)^{(n+1)(\frac{\sigma r}{2}+\frac{t}{2}-s-\frac{\theta_1}{q}
+\theta_2\frac{r}{p}-\frac{r\sigma}{2})}\\
&\preceq K\nu\bigl(B_D(a,\rho)\bigr)^{\frac{t}{2}-s-\frac{\theta_1}{q}
+\theta_2\frac{r}{p}}\;.
\end{aligned}
\]
Since $\mu$ is a finite measure, a similar estimate holds when $\delta(a)\ge\delta_1$. 
Then Theorem~\ref{carthetaCarluno} implies that $\mu$ is $(\lambda,\gamma)$-skew Carleson 
with $\|\mu\|_{\lambda,\gamma}\preceq K$ as claimed.

Now let us assume $0<\lambda<1$. Assume first $r=1$. Choose $\beta\in\mathbb{N}$ with $\beta\ge t$ and set
\[
\sigma=\theta_1+q\left(\beta-\frac{t}{2}+s\right)>1-\frac{1}{n+1}\;.
\]
We can apply \eqref{eqltretreu} and Lemma~\ref{th:fromStoTnew}.(i) with $p=q$ to get
\[
\|T^\beta_\mu f\|_{q,(n+1)(\sigma-1)}\preceq K \|f\|_{p,(n+1)(\theta_2-1)}\;.
\]
Therefore Proposition~\ref{th:TtoCarl} implies that $\mu$ is $(\lambda,\tilde\gamma)$-skew Carleson with
\[
\tilde\gamma=\frac{1}{\lambda}\left(\beta+\frac{\theta_2}{p}-\frac{\sigma}{q}\right)=\frac{1}{\lambda}\left(\frac{t}{2}+\frac{\theta_2}{p}-\frac{\theta_1}{q}-s\right)=\gamma\;,
\]
and $\|\mu\|_{\lambda,\gamma}\preceq K$ as claimed.

Assume now $r>1$, and choose $\alpha>0$ so that
\[
\beta=\frac{t}{r}+\frac{\alpha}{r'}>\frac{1}{p}\max\left\{1,\theta_2,p-1+\theta_2\right\}
\]
and $\beta\in\N$. We also require that 
$\alpha$ is such that $\alpha>\lambda\gamma+\frac{n}{n+1}(1-\lambda)$ and
\[
\sigma:=\tau\left[\frac{\alpha-\lambda\gamma}{r'}+\frac{1}{qr}\left(\theta_1+q\left(s+\frac{t}{2}\right)\right)\right]>1-\frac{1}{n+1}\;,
\]
where 
\[
\tau=\frac{1}{1-\lambda+\frac{1}{p}}\;.
\]
Assume for a moment that $\mu$ has compact support. Then $\|\delta^{-(n+1)(\gamma-\frac{\alpha}{2})}B^\alpha\mu\|_{\frac{1}{1-\lambda},(n+1)(\gamma-1)}$ is finite;
therefore \eqref{eqltretreu} and Corollary~\ref{th:fromStoTfin} applied with $\tilde p=p$ imply that $T^\beta_\mu$ is bounded from $A^p\bigl(D,(n+1)(\theta_2-1)\bigr)$ to $A^\tau\bigl(D,(n+1)(\sigma-1)\bigr)$,
with 
\[
\|T^\beta_\mu\|\preceq K\|\delta^{-(n+1)(\gamma-\frac{\alpha}{2})}B^\alpha\mu\|_{\frac{1}{1-\lambda},(n+1)(\gamma-1)}^{1/r'}\;.
\]
Proposition~\ref{th:TtoCarl} then yields that $\mu$ is $(\tilde\lambda,\tilde\gamma)$-skew Carleson with
\[
\tilde\lambda=1+\frac{1}{p}-\frac{1}{\tau}=\lambda
\]
and
\[
\begin{aligned}
\tilde\gamma&=\frac{1}{\lambda}\left(\beta+\frac{\theta_2}{p}-\frac{\sigma}{\tau}\right)=\frac{1}{\lambda}\left(\beta+\frac{\theta_2}{p}-\frac{\alpha}{r'}+\frac{\lambda\gamma}{r'}-
\frac{\theta_1}{qr}-\frac{1}{r}\left(s+\frac{t}{2}\right)\right)\\
&=\frac{1}{\lambda}\left(\frac{t}{2r}+\frac{\theta_2}{p}+\frac{t}{2r'}+\frac{\theta_2 r}{pr'}-\frac{\theta_1}{qr'}-\frac{s}{r'}-\frac{\theta_1}{qr}-\frac{s}{r}\right)\\
&=\frac{1}{\lambda}\left(\frac{t}{2}-\frac{\theta_1}{q}-s+\frac{\theta_2 r}{p}\right)=\gamma\;.
\end{aligned}
\]
Furthermore, we also have
\[
 \|\delta^{-(n+1)(\gamma-\frac{\alpha}{2})}B^\alpha\mu\|_{\frac{1}{1-\lambda},(n+1)(\gamma-1)}\preceq K \|\delta^{-(n+1)(\gamma-\frac{\alpha}{2})}B^\alpha\mu\|_{\frac{1}{1-\lambda},(n+1)(\gamma-1)}^{1/r'}
\]
and thus
\[
 \|\delta^{-(n+1)(\gamma-\frac{\alpha}{2})}B^\alpha\mu\|_{\frac{1}{1-\lambda},(n+1)(\gamma-1)}\preceq K\;.
\]
An easy limit argument then shows that this holds even when the support of $\mu$ is not compact, and then, by Theorem~\ref{carthetaCarldue}, $\mu$ is $(\lambda,\gamma)$-skew Carleson 
with $\|\mu\|_{\lambda,\gamma}\preceq K$.

We are left with the case $0<r<1$. Choose $R>1$ and set $\mu^*=\delta^A\mu$, with
\[
A=(n+1)\frac{(R-r)\theta_2}{p}\;.
\]
First of all, fix $r_0\in(0,1)$ and set $R_0=\frac{1}{2}(1+r_0)$. Then, for any $z\in D$, Lemmas~\ref{sei}, \ref{sette} and \ref{due} yield
\[
\begin{aligned}
|f(z)|^p&\preceq \frac{1}{\nu\bigl(B_D(z,r_0)\bigr)}\int_{B_D(z,R_0)}|f(w)|^p\,d\nu(w)
\preceq \delta(z)^{-(n+1)\theta_2}\int_{B_D(z,R_0)}|f(w)|^p\delta(w)^{(n+1)(\theta_2-1)}\,d\nu(w)\\
&\le \delta(z)^{-(n+1)\theta_2}\|f\|^p_{p,(n+1)(\theta_2-1)}\;.
\end{aligned}
\]
Then \eqref{eqltretreu} yields
%
\begin{equation*}
\begin{aligned}
\|S^{s,R}_{\mu^*,t}&f\|_{q,(n+1)(\theta_1-1)}^q\\
&=\int_D\left(\delta(z)^{(n+1)s}\int_D |k_z(w)|^t |f(w)|^r |f(w)|^{R-r}\,d\mu^*(w)
\right)^q \delta(z)^{(n+1)(\theta_1-1)}\,d\nu(z)\\
&\preceq\|f\|^{q(R-r)}_{p,(n+1)(\theta_2-1)}\int_D\!\!\!\left(\!\delta(z)^{(n+1)s}\!\!\int_D\!\! |k_z(w)|^t |f(w)|^r \delta(w)^{A-(n+1)\frac{(R-r)\theta_2}{p}}d\mu(w)\!\right)^q
\!\!\!\delta(z)^{(n+1)(\theta_1-1)}d\nu(z)\\
&=\|f\|^{q(R-r)}_{p,(n+1)(\theta_2-1)}\|S^{s,r}_{\mu,t}f\|^q_{q,(n+1)(\theta_1-1)}\\
&\le K \|f\|^{qR}_{p,(n+1)(\theta_2-1)}
\end{aligned}
\end{equation*}
for all $f\in A^p\bigl(D,(n+1)(\theta_2-1)\bigr)$. Arguing as before, we can prove that 
$\mu^*$ is $(\lambda,\gamma^*)$-skew Carleson with $\|\mu^*\|_{\lambda,\gamma^*}\preceq K$, where 
\[
\gamma^*=\frac{1}{\lambda}
\left(\frac{t}{2}+\frac{\theta_2R}{p}-\frac{\theta_1}{q}-s\right)=\gamma+\frac{(R-r)\theta_2}{\lambda p}\;.
\]
But $\mu=\delta^{-(n+1)\frac{(R-r)\theta_2}{p}}\mu^*$; then Lemma~\ref{th:last} implies that $\mu$ is
$(\lambda,\gamma)$-skew Carleson with $\|\mu\|_{\lambda,\gamma}\preceq K$, and we are done.
\end{proof}

We finally have all the ingredients to prove our main result.

\begin{proof}[Proof of Theorem \ref{maintheorem}]
Assume that $\mu$ is $(\lambda,\gamma)$-skew Carleson. For $k=1$ the assertion is just the definition of $(\lambda,\gamma)$-skew Carleson; so we can assume $k\ge 2$.

For $j=1,\ldots,k$ put $\beta_j=\lambda\frac{p_j}{q_j}$. Then we have $\beta_j>1$, $\frac{q_j}{p_j}\beta_j=\lambda$, 
and
\[
\sum_{j=1}^k \frac{1}{\beta_j}=1\;.
\]
%
Now define $\eta_j\in\R$ as
\[
\eta_j=\frac{q_j}{p_j}\theta_j-\frac{1}{\beta_j}\lambda\gamma=\frac{q_j}{p_j}(\theta_j-\gamma)\;;
\]
in particular
\begin{equation}
\gamma+\frac{1}{\lambda}\beta_j\eta_j=\theta_j\;.
\label{eq:ga}
\end{equation}
It is easy to check that $\eta_1+\cdots+\eta_k=0$; then 
H\"older's inequality yields
\begin{equation}
\int_D \prod_{j=1}^k |f_j(z)|^{q_j}\,d\mu(z)\le\prod_{j=1}^k\left[\int_D |f_j(z)|^{\beta_jq_j}\delta(z)^{\beta_j\eta_j}\,d\mu(z)\right]^{1/\beta_j}\;.
\label{eq:Cuno}
\end{equation}
Now, Lemma~\ref{th:last} implies that $\delta^{\beta_j\eta_j}\mu$ is $(\lambda,\gamma+\frac{1}{\lambda}\beta_j\eta_j)$-skew Carleson, that is, $(\lambda,\theta_j)$-skew Carleson,  by \eqref{eq:ga}. But $\lambda=\frac{q_j\beta_j}{p_j}$; hence Theorems~\ref{carthetaCarluno} and~\ref{carthetaCarldue} imply that $\delta^{\beta_j\eta_j}\mu$ is $(p_j,q_j\beta_j;\alpha_j)$-skew Carleson, with $\alpha_j=(n+1)(\theta_j-1)$. Therefore
%
\begin{equation}
\left[\int_D |f_j(z)|^{\beta_j q_j}\delta(z)^{\beta_j\eta_j}\,d\mu(z)\right]^{1/\beta_j}\preceq \|f_j\|^{q_j}_{p_j,(n+1)(\theta_j-1)}
\label{eq:Cdue}
\end{equation}
for $j=1,\ldots, k$, and \eqref{eq:main} is proved (see also Remark~\ref{rem:finale} below). 

Assume now that \eqref{eq:main} holds for any $f_j\in A^{p_j}\bigl(D,(n+1)(\theta_j-1)\bigr)$ with $j=1,\ldots,k$; we would like to prove by induction that $\mu$ is a $(\lambda,\gamma)$-skew Carleson measure with $\|\mu\|_{\lambda,\gamma}\preceq C$. If $k=1$ there is nothing to prove, so we can assume $k\ge 2$.

Assume first $\lambda\ge 1$, and let $\alpha_j=(n+1)(\theta_j-1)$ for $j=1,\ldots,k$. Choose $\sigma_1,\ldots,\sigma_k\in\N^*$ such that
\[
p_j\sigma_j>\max\{1,\theta_j\}
\]
for all $j=1,\ldots,k$, and
\[
\sum_{j=1}^k q_j\sigma_j>\lambda\gamma\;,
\]
and set
\[
r_j=(n+1)\left[\frac{\sigma_j}{2}-\frac{\theta_j}{p_j}\right]
\]
for all $j=1,\ldots,k$. 

For any $a\in D$ and $j=1,\ldots,k$, consider
\[
f_{j,a}(z)=k_a(z)^{\sigma_j}\delta(a)^{r_j}\;.
\]
Then, since $\alpha_j<(n+1)(p_j\sigma_j-1)$ by the choice of~$\sigma_j$, applying 
Proposition~\ref{intest} we obtain
\[
\|k_a^{\sigma_j}\|_{p_j,\alpha_j}=\|k_a\|^{\sigma_j}_{p_j\sigma_j,\alpha_j}\preceq
\delta(a)^{\frac{1}{p_j}\left[\alpha_j-(n+1)\left(\frac{p_j\sigma_j}{2}-1\right)\right]}
=\delta(a)^{-r_j}\;,
\]
and hence
\[
\|f_{j,a}\|_{p_j,\alpha_j}\preceq 1
\]
for $j=1,\ldots,k$. Thus \eqref{eq:main} yields
\begin{equation}
\int_D\prod_{j=1}^k |f_{j,a}(z)|^{q_j}\,d\mu(z)\le C \prod_{j=1}^k\|f_{j,a}\|^{q_j}_{p_j,\alpha_j}
\preceq C\;.
\label{eq:mdue}
\end{equation}
Now recall that
\[
\prod_{j=1}^k |f_{j,a}(z)|^{q_j}=|k_a(z)|^{\sum_j q_j\sigma_j}\delta(a)^{\sum_j q_jr_j}\;.
\]
We have
\[
\sum_{j=1}^k q_jr_j=(n+1)\sum_{j=1}^k\left[\frac{q_j\sigma_j}{2}-\theta_j\frac{q_j}{p_j}\right]
=\frac{n+1}{2}\sum_{j=1}^k q_j\sigma_j-(n+1)\lambda\gamma\;,
\]
so, setting $s=\sum_j \sigma_jq_j$, \eqref{eq:mdue} becomes
\[
\delta^{(n+1)\left(\frac{s}{2}-\lambda\gamma\right)}B^s\mu\preceq C\;,
\]
and Theorem~\ref{carthetaCarluno} implies that $\mu$ is $(\lambda,\gamma)$-Carleson with
\[
\|\mu\|_{\lambda,\gamma} \approx \|\delta^{(n+1)\left(\frac{s}{2}-\lambda\gamma\right)}B^s\mu\|_\infty \preceq C\;.
\]
%

We are left with the case $0<\lambda<1$. 
We argue again by induction on~$k$. If $k=1$, it is the definition of skew Carleson measure; so assume the assertion holds for $k-1$. Set 
\[
\tilde\lambda=\sum_{j=1}^{k-1}\frac{q_j}{p_j}\qquad\mathrm{and}\qquad
\tilde\gamma=\frac{1}{\tilde\lambda}\sum_{j=1}^{k-1}\theta_j\frac{q_j}{p_j}\;.
\]
Fix a function $g\in A^{p_k}\bigl(D,(n+1)(\theta_k-1)\bigr)$, and set $\mu_k=|g|^{q_k}\mu$. Then 
\eqref{eq:main} yields
\[
\int_D \prod_{j=1}^{k-1}|f_j(z)|^{q_j}\,d\mu_k(z)\le C \|g\|^{q_k}_{p_k,(n+1)(\theta_k-1)}
\prod_{j=1}^{k-1}\|f_j\|^{q_j}_{p_j,(n+1)(\theta_j-1)}
\]
for all $f_j\in A^{p_j}\bigl(D,(n+1)(\theta_j-1)\bigr)$ with $j=1,\ldots,k-1$. By induction, this means that $\mu_k$ is a $(\tilde\lambda,\tilde\gamma)$-skew Carleson measure with $\|\mu_k\|_{\tilde\lambda,\tilde\gamma}\preceq C\|g\|^{q_k}_{p_k,(n+1)(\theta_k-1)}$. Since $\tilde\lambda<\lambda<1$, and $\tilde\gamma>1-\frac{1}{n+1}$, Theorem~\ref{carthetaCarldue} implies that $\delta^{-(n+1)(\tilde\gamma-\frac{t}{2})}B^t\mu_k\in L^{1/(1-\tilde\lambda)}
\bigl(D, (n+1)(\tilde\gamma-1)\bigr)$ for all $t>\tilde\lambda\tilde\gamma+\frac{n}{n+1}(1-\tilde\lambda)$, with
\[
\left\|\delta^{-(n+1)(\tilde\gamma-\frac{t}{2})}B^t\mu_k\right\|_{1/(1-\tilde\lambda), (n+1)(\tilde\gamma-1)}\preceq C\|g\|^{q_k}_{p_k,(n+1)(\theta_k-1)}\;.
\]
Writing explicitely the previous formula we obtain
\[
\left[\int_D\left[\int_D |k_a(z)|^t|g(z)|^{q_k}\,d\mu(z)\right]^{1/(1-\tilde\lambda)}
\delta(a)^{-\frac{n+1}{1-\tilde\lambda}(\tilde\gamma-\frac{t}{2})}\delta(a)^{(n+1)(\tilde\gamma-1)}
\,d\nu(a)\right]^{1-\tilde\lambda}
\!\!\!\!\!
\preceq C\|g\|^{q_k}_{p_k,(n+1)(\theta_k-1)}\;,
\]
that is
\[
\|S^{s,q_k}_{\mu,t}g\|_{1/(1-\tilde\lambda),(n+1)(\tilde\gamma-1)}
\preceq C\|g\|^{q_k}_{p_k,(n+1)(\theta_k-1)}\;,
\]
where $s=\frac{t}{2}-\tilde\gamma$. Choosing $t>\frac{1}{p_k}\max\{1,\theta_k,p_k-1-\theta_k\}$ such that $s>0$, 
we deduce from Proposition~\ref{lemmatretre} that $\mu$ is a $(\lambda^*,\gamma^*)$-skew Carleson measure with $\|\mu\|_{\lambda^*,\gamma^*}\preceq C$, where
\[
\lambda^*=1+\frac{q_k}{p_k}-(1-\tilde\lambda)=\lambda\quad\mathrm{and}\quad
\gamma^*=\frac{1}{\lambda^*}\left(\theta_k\frac{q_k}{p_k}+\tilde\gamma\tilde\lambda\right)=\gamma\;,
\]
and we are done.
\end{proof}

\begin{remark}
\label{rem:finale}
If $\mu$ is a $(\lambda,\gamma)$-skew Carleson measure, we can estimate the constant $C$ in \eqref{eq:main}. Fix $r\in(0,1)$. Then Lemmas~\ref{sei} and~\ref{sette} yield 
\[
(\widehat{\delta^{\beta_j\eta_j}\mu})_r \approx \delta^{(n+1)\beta_j\eta_j}\hat\mu_r\;.
\]
If $\lambda\ge 1$ we can now use \eqref{norme} to get
\[
\begin{aligned}
\|\delta^{\beta_j\eta_j}\mu\|_{p_j,q_j\beta_j;\alpha_j}
&\approx \|\delta^{-(n+1)(\lambda\theta_j-1)}(\widehat{\delta^{\beta_j\eta_j}\mu})_r \|_\infty\\
&\approx \|\delta^{-(n+1)(\lambda\theta_j-1-\beta_j\eta_j)}\hat\mu_r\|_\infty
= \|\delta^{-(n+1)(\lambda\gamma-1)}\hat\mu_r\|_\infty
\approx \|\mu\|_{\lambda,\gamma}\;.
\end{aligned}
\]
Analogously, if $0<\lambda<1$ we can use \eqref{norme2} to get
\[
\begin{aligned}
\|\delta^{\beta_j\eta_j}\mu\|_{p_j,q_j\beta_j;\alpha_j}
&\approx \|\delta^{-(n+1)(\theta_j-1)\lambda}(\widehat{\delta^{\beta_j\eta_j}\mu})_r \|_{\frac{1}{1-\lambda}}\\
&\approx \|\delta^{-(n+1)(\lambda\theta_j-\lambda-\beta_j\eta_j)}\hat\mu_r\|_{\frac{1}{1-\lambda}}
= \|\delta^{-(n+1)(\gamma-1)\lambda}\hat\mu_r\|_{\frac{1}{1-\lambda}}
\approx \|\mu\|_{\lambda,\gamma}\;.
\end{aligned}
\]
Therefore in both cases \eqref{eq:Cuno} and \eqref{eq:Cdue} yield 
\[
C\approx \|\mu\|_{\lambda,\gamma}^{\sum_j q_j}\;.
\]
\end{remark}

\end{document}